\documentclass{amsart}
\usepackage{mathrsfs,amssymb,mathrsfs, multirow,setspace}
\usepackage[a4paper, total={7in, 9in}]{geometry}
\usepackage{enumerate}
\theoremstyle{plain}
\usepackage{graphicx}
\usepackage{mathtools}

\newtheorem{theorem}{Theorem}[section]
\newtheorem{lemma}[theorem]{Lemma}
\newtheorem{proposition}[theorem]{Proposition}

\theoremstyle{definition}

\theoremstyle{remark}

\input xy
\xyoption{all}

\begin{document}

	\title[Characterization of rings with genus two prime ideal sum graphs]{Characterization of rings with genus two prime ideal sum graphs}
	
	\author[Praveen Mathil, Jitender Kumar]{Praveen Mathil, Jitender Kumar*}
	\address{Department of Mathematics, Birla Institute of Technology and Science Pilani, Pilani-333031, India}
	\email{maithilpraveen@gmail.com, jitenderarora09@gmail.com}

	%\date{...}
\begin{abstract}
Let $R$ be a commutative ring with unity. The prime ideal sum graph of the ring $R$ is a simple undirected graph whose vertex set is the set of nonzero proper ideals of $R$ and two distinct vertices $I$ and $J$ are adjacent if and only if $I + J$ is a prime ideal of $R$. In this paper, we characterize all the finite non-local commutative rings whose prime ideal sum graph is of genus $2$.
\end{abstract}

\subjclass[2020]{05C25}
\keywords{Non-local ring, ideals, genus a graph.\\ *  Corresponding author}
\maketitle

\section{Historical Background and Main Result}
Algebraic graph theory is a branch of mathematics that studies algebraic structures using techniques from graph theory and graphs from an algebraic perspective. In the last three decades, research on algebraic graphs has grown significantly, with numerous advances in theory and applications. It has numerous applications in computer science including computer networks, image processing, and data mining.  Many researchers attached graphs to various algebraic structure, viz. semigroups, groups, rings, vector spaces, etc., and studied their properties through graphs. Various algebraic graphs associated to ring structures have been defined and investigated by the researchers (see, \cite{afkhami2012generalized, anderson2008total, anderson1999zero, behboodi2011annihilating, biswas2022subgraph, maimani2008comaximal, nazim2022essential, nazim2023some}). Because of the vital role of ideals in the theory of rings, numerous authors have studied the graphs on rings with respect to their ideals such as inclusion ideal graph \cite{akbari2015inclusion}, intersection graphs of ideals \cite{chakrabarty2009intersection}, ideal-relation graph \cite{ma2016automorphism}, prime ideal sum graph \cite{saha2022prime}, reduced cozero-divisor graph \cite{wilkens2011reduced}, co-maximal ideal graph \cite{ye2012co} etc.

The genus $g( \Gamma)$ of a graph $\Gamma$ is a topological invariant that measures the complexity of its embedding on a surface. It is defined as the minimum number of handles (i.e., topological spheres with holes) that are needed to embed the graph on a surface without edge crossings. The genus of a graph is a fundamental concept in topology, and has important applications in physics and computer science. Determining the genus of a graph is a fundamental but challenging problem. Indeed, it is NP-complete. Several authors have considered the problem of determining the genus of various algebraic graphs. All genus one graphs that are intersection graphs of rings are classified in \cite{pucanovic2014toroidality}. Moreover, an improvement over the existing results concerning the planarity of intersection graphs of rings is also presented. Further, Pucanovi\'{c} \emph{et al.} \cite{pucanovic2014genus} classified the commutative rings for which the intersection graph of ideals has genus two. All the rings with genus one (or two) reduced cozero-divisor graphs have been classified in \cite{jesili2022genus, mathil2022characterization}. More work related to the genus of graphs defined on rings can be found in \cite{akbari2003zero, asir2020classification, dorbidi2016some, kalaimurugan2022genus, maimani2012rings,rehman2022planarity,nazim2023genus, selvakumar2022genus, krishnan2018classification, chelvam2013genus, wang2006zero, wickham2008classification, wickham2009rings}.

 The prime ideal sum graph of a commutative ring was introduced by Saha \emph{et al.} \cite{saha2022prime}. The prime ideal sum graph $\text{PIS}(R)$ of a commutative $R$ is a simple undirected graph whose vertices are nonzero proper ideals of $R$ and two distinct vertices $I$ and $J$ are adjacent if and only if $I + J$ is a prime ideal of $R$. In this connection, the authors explored the graph theoretic properties of $\text{PIS}(R)$ such as clique number, chromatic number, and the domination number of $\text{PIS}(R)$. Further, the planar graphs which may occur as prime ideal sum graph of non-local commutative rings have been studied in \cite{mathil2022embedding}. Moreover, together with forbidden graph classes classification of $\text{PIS}(R)$, Mathil \emph{et al.} \cite{mathil2022embedding} have characterized the non-local commutative rings $R$ such that $\text{PIS}(R)$ is of genus one. In this paper, we continue the research on prime ideal sum graphs. The aim of this manuscript is to characterize the finite non-local commutative rings $R$ such that the genus of $\text{PIS}(R)$ is two. The main result of this paper is as follows.

\begin{theorem}\label{genus2R_1R_2}
Let $R \cong R_1 \times R_2 \times \cdots \times R_n$ ($n \ge 2$) be a non-local commutative ring, where each $R_i$ is a local ring with maximal ideal $\mathcal{M}_i$. Then $g(\textnormal{PIS}(R)) = 2$ if and only if one of the following holds:
\begin{enumerate}[\rm(i)]
    \item $R \cong R_1 \times F_2$ such that $\mathcal{M}_1 = \langle x,y \rangle$ and $x^2 = y^2 = 0$.
    
    \item $R \cong  R_1 \times R_2$, where $R_1$ and $R_2$ are principal rings such that $\eta(\mathcal{M}_1) \in \{4,5 \}$ and $\eta(\mathcal{M}_2) =2$.
    \end{enumerate}
\end{theorem}

  \section{preliminaries}
In this section, we recall some graph theoretic notions from  \cite{westgraph, white1985graphs}. A \emph{graph} $\Gamma$ is an ordered pair $(V(\Gamma), E(\Gamma))$, where $V(\Gamma)$ is the set of vertices and $E(\Gamma)$ is the set of edges of $\Gamma$. Two distinct vertices $u, v \in V(\Gamma)$ are $\mathit{adjacent}$ in $\Gamma$, denoted by $u \sim v$ (or $(u,  v)$), if there is an edge between $u$ and $v$. Otherwise, we write as $u \nsim v$. A graph $\Gamma' = (V(\Gamma'), E(\Gamma'))$ is said to be a \emph{subgraph} of the graph $\Gamma$ if $V(\Gamma') \subseteq V(\Gamma)$ and $E(\Gamma') \subseteq E(\Gamma)$. If $X \subseteq V(\Gamma)$ then the subgraph $\Gamma(X)$ induced by $X$ is the graph with vertex set $X$ and two vertices of $\Gamma(X)$ are adjacent if and only if they are adjacent in $\Gamma$. A graph $\Gamma$ is said to be \emph{complete} if any two vertices are adjacent in $\Gamma$. The complete graph on $n$ vertices is denoted by $K_n$. The \emph{complete bipartite graph} $K_{m,n}$ is a graph $\Gamma$ if the vertex set $V(\Gamma)$ of $\Gamma$ can be partitioned into two nonempty sets $A$ and $B$, where $|A|= m$ and $|B| = n$, such that two distinct vertices of $\Gamma$ are adjacent if and only if they belong to different partition sets.

 The \emph{subdivision} of an edge $(u,v)$ in a graph $\Gamma$ is an operation that deletes the edge $(u,v)$ from $\Gamma$ and add two new edges $(u,w)$ and $(w,v)$ along with a new vertex $w$. The \emph{subdivision of $\Gamma$}, is a graph obtained from $\Gamma$ by a sequence of edge subdivisions. The \emph{edge contraction} of the edge $(u,v)$ is a single vertex, denoted by $[u,v]$, which is obtained by merging the endpoints $u$ and $v$. Also, $[u,v]$ is adjacent to all the edges that were adjacent to $u$ or $v$. Two graphs are \emph{homeomorphic} if both can be obtained from the same graph by subdivision or contracting of edges. A graph $\Gamma$ is said to be \emph{planar} if it can be drawn on a plane without edge crossings.

A compact connected topological space such that each point has a neighbourhood homeomorphic to an open disc is called a surface. For a non-negative integer $g$, let $\mathbb{S}_{g}$ be the orientable surface with $g$ handles. The genus $g(\Gamma)$ of a graph $\Gamma$ is the minimum integer $g$ such that the graph can be embedded in $\mathbb{S}_{g}$ i.e., the graph $\Gamma$ can be drawn into a surface $\mathbb{S}_{g}$ with no edge crossing. Note that the graphs having genus $0$ are planar and the graphs having genus one are toroidal. The following results are useful in the sequel.

\begin{proposition}{\cite[Ringel and Youngs]{white1985graphs}}
\label{genus}
 Let $m, n$ be positive integers. 
 \begin{enumerate}
     \item[{\rm(i)}]If $n \ge 3$, then $g(K_n) = \left\lceil \frac{(n-3)(n-4)}{12} \right\rceil$.
     \item[{\rm(ii)}]If $m, n\geq 2$, then $g(K_{m,n}) =  \left\lceil \frac{(m-2)(n-2)}{4}  \right\rceil$.
 \end{enumerate}
 \end{proposition}

\begin{lemma}\cite{white2001graphs}
\label{genusofblocks}
The genus of a connected graph $\Gamma$ is the sum of the genera of its blocks.
\end{lemma}

 Throughout the paper, the ring $R$ is a finite non-local commutative ring with unity and $F_i$ denotes a finite field. For basic definitions of ring theory, we refer the reader to  \cite{atiyah1994introduction}. By $\mathcal{I}^*(R)$, we mean the set of non-zero proper ideals of $R$. The \emph{nilpotency index} $\eta(I)$ of an ideal $I$ of $R$ is the smallest positive integer $n$ such that $I^n = 0$. A ring $R$ is said to be \emph{local} if it has a unique maximal ideal. 
 By the structural theorem \cite{atiyah1994introduction}, a finite non-local commutative ring $R$ is uniquely (up to isomorphism) a finite direct product of local rings $R_i$ that is $R \cong R_1 \times R_2 \times \cdots \times R_n$, where $n \geq 2$. The following results of the prime ideal sum graph are useful in the sequel.

%%%%%%%%%%%%%%%%%%%%%%%%%%%%%%%%%%%%%%%%%%%%%%%%%%%%%%%%%%%%%%%%%%%%%%%%%%%%%%%%%%%%%%%%%%%%%%%%%%%%%%%%%%%%%%%%%%%%%%%%%%%%%%%%%%%%%%%%%%%%%%%%%%%%%% 

\begin{theorem}{\cite[Theorem 4.1]{mathil2022embedding}}\label{Planar_primeidealsum}
 Let $R \cong R_1 \times R_2 \times \cdots \times R_n$ ($n \geq 2$) be a non-local commutative ring, where each $R_i$ is a local ring with maximal ideal $\mathcal{M}_i$. Then the graph $\textnormal{PIS}(R)$ is planar if and only if one of the following holds:
 \begin{enumerate}[\rm(i)]
    \item $R \cong  F_1 \times F_2 \times F_3$;
    
    \item $R \cong  F_1 \times F_2$;
    \item $R \cong  R_1 \times R_2$ such that $\mathcal{I}^*(R_1) = \{ \mathcal{M}_1 \}$ and  $\mathcal{I}^*(R_2) = \{ \mathcal{M}_2 \}$;
    \item $R \cong  F_1 \times R_2$, where the local ring $R_2$ is principal.
    
    \end{enumerate}
\end{theorem}

\begin{theorem}{\cite[Theorem 4.4]{mathil2022embedding}}\label{genus13product}
Let $R \cong R_1 \times R_2 \times \cdots \times R_n$ ($n \geq 3$) be a non-local commutative ring, where each $R_i$ is a local ring with maximal ideal $\mathcal{M}_i$. Then $g(\textnormal{PIS}(R)) = 1$ if and only if $R \cong R_1 \times F_2 \times F_3$ such that $\mathcal{I}^{*}(R_1) = \{\mathcal{M}_1\}$.
\end{theorem}

\begin{theorem}{\cite[Theorem 4.5]{mathil2022embedding}}\label{genusofR1R2}
  Let $R \cong R_1 \times R_2$, where $R_1$ and $R_2$ are local rings with maximal ideals $\mathcal{M}_1$ and $\mathcal{M}_2$, respectively. Then $g(\textnormal{PIS}(R)) = 1$ if and only if $\mathcal{I}^*(R_1) = \{I_1, \mathcal{M}_1\}$ and $\mathcal{I}^*(R_2) = \{ \mathcal{M}_2\}$.
\end{theorem}

\section{Proof of the Main Result}

In this section, we prove the main result of the paper. Let $i, j, k \in \{1, 2, \ldots, n\}$. For $F_1 \times F_2 \times \cdots \times F_n$ by $F_{ijk\ldots}$, we mean $\underbrace{\langle 0 \rangle \times \cdots \times \langle 0 \rangle}_{i-1 \ \text{times}} \times F_i \times \underbrace{\langle 0 \rangle \times \cdots \times \langle 0 \rangle}_{j-i-1 \ \text{times}} \times F_j \times \underbrace{\langle 0 \rangle \times \cdots \times\langle 0 \rangle}_{k-j-1 \ \text{times}} \times F_k \times \underbrace{\langle 0 \rangle \times \cdots \times \langle 0 \rangle}_{n-k \ \text{times}}$. The following proposition is essential to prove the Theorem \ref{genus2R_1R_2}.

\begin{proposition}\label{genus2ofnproduct}
    Let $R \cong R_1 \times R_2 \times \cdots \times R_n$($n \geq3$) be a non-local commutative ring. Then $g(\textnormal{PIS}(R)) \neq 2$.
    
\end{proposition}

\begin{proof}
  Let $R \cong R_1 \times R_2 \times \cdots \times R_n$ be a non-local commutative ring, where each $R_i$ is a local ring with maximal ideal $\mathcal{M}_i$. First suppose that $n \ge 5$. To prove the result it is sufficient to show that $g(\textnormal{PIS}(F_1 \times F_2 \times F_3 \times F_4 \times F_5)) \neq 2$. By Figure \ref{subdivisionK_55fig1}, $\textnormal{PIS}(F_1 \times F_2 \times F_3 \times F_4 \times F_5)$ contains a subgraph homeomorphic to $K_{5,5}$. 
        \begin{figure}[h!]
\centering
\includegraphics[width=0.5 \textwidth]{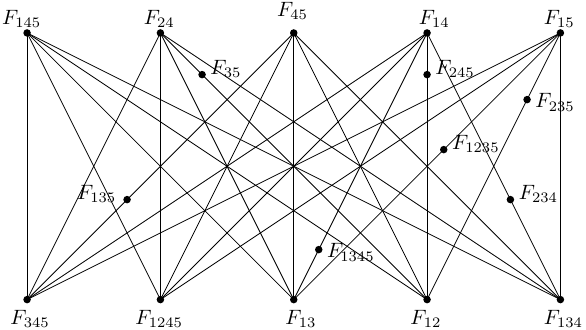}
\caption{Subgraph of $\text{PIS}(F_1 \times F_2 \times F_3 \times F_4 \times F_5)$}
\label{subdivisionK_55fig1}
\end{figure}
    Thus, by Proposition \ref{genus}, the results holds for $n \ge 5$. Now we assume that $R \cong R_1 \times R_2 \times R_3 \times R_4$. We show that $g(\textnormal{PIS}(F_1 \times F_2 \times F_3 \times F_4 )) \ge 3$.
Consider $R' = F_1 \times F_2 \times F_3 \times F_4$. Let $H$ be the subgraph of $\textnormal{PIS}(R)$ shown in Figure \ref{K_33subdivisionF1234}. Note that $H$ is homeomorphic to $K_{3,3}$.
            \begin{figure}[h!]
\centering
\includegraphics[width=0.4 \textwidth]{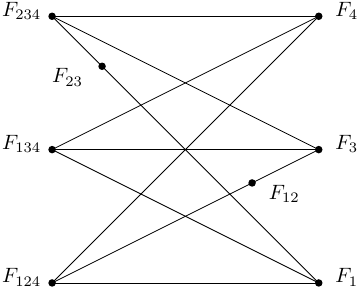}
\caption{Subgraph of $\text{PIS}(F_1 \times F_2 \times F_3 \times F_4)$}
\label{K_33subdivisionF1234}
\end{figure}
    Consequently, $g(H) = 1$ and so $g(\textnormal{PIS}(R')) \ge 1$. Suppose that $g(\textnormal{PIS}(R')) = 1$. Now to embed $\textnormal{PIS}(R')$ in $\mathbb{S}_1$ using $H$, first we insert the vertices $F_{24}$, $F_{34}$, $F_{14}$, $F_2$ and their incident edges in an embedding of $H$ in $\mathbb{S}_1$. Note that $F_{24} \sim F_{34} \sim F_{14} \sim F_2$ in $\textnormal{PIS}(R')$. All these vertices must be inserted in the same face $F$. Note that both the vertices $[F_{24}, F_{34}]$ and $[F_{14}, F_2]$ are adjacent to the vertices $F_{124}$, $F_{134}$ and $F_{234}$. It implies that the insertion of the vertices $F_{24}$, $F_{34}$, $F_{14}$ and $F_2$ in $F$ is not possible without edge crossings. Let $H'$ be the graph induced by the vertex set $V(H) \cup \{ F_{24}, F_{34}, F_{14}, F_2 \} $. It implies that $g(H') \ge 2$ and so $g(\textnormal{PIS}(R')) \ge 2$. Suppose that $g(H') = g(\textnormal{PIS}(R')) = 2$. To embed $\textnormal{PIS}(R')$ in $\mathbb{S}_2$, first we insert the vertices $F_{13}$, $F_{123}$ and their incident edges in an embedding of $H'$ in $\mathbb{S}_2$. Since $F_{13} \sim F_{123}$ in $\textnormal{PIS}(R')$, they must be inserted in the same face $F'$. Note that both the vertices $F_{13}$ and $F_{123}$ are adjacent with $F_{12}$, $F_{23}$ and $F_2$. It follows that the insertion of $F_{13}$ and $F_{123}$ in $F'$ is not possible without edge crossings. Therefore, $g(\textnormal{PIS}(R')) \ge 3$ and so $g(\textnormal{PIS}(R_1 \times R_2 \times R_3 \times R_4)) \ge 3$.
    
    We may now assume that $n =3$ i.e., $R \cong R_1 \times R_2 \times R_3$. Suppose that one of $R_i$ is a field. Without loss of generality assume that $R_3$ is a field so that $R \cong R_1 \times R_2 \times F_3$. By Figure \ref{K_55subdivisionR_1R_2F_3fig1}, $\textnormal{PIS}(R)$ contains a subgraph homeomorphic to $K_{5,5}$. It follows that $g(\textnormal{PIS}(R_1 \times R_2 \times F_3)) \ge 3$.
\begin{figure}[h!]
\centering
\includegraphics[width=0.75 \textwidth]{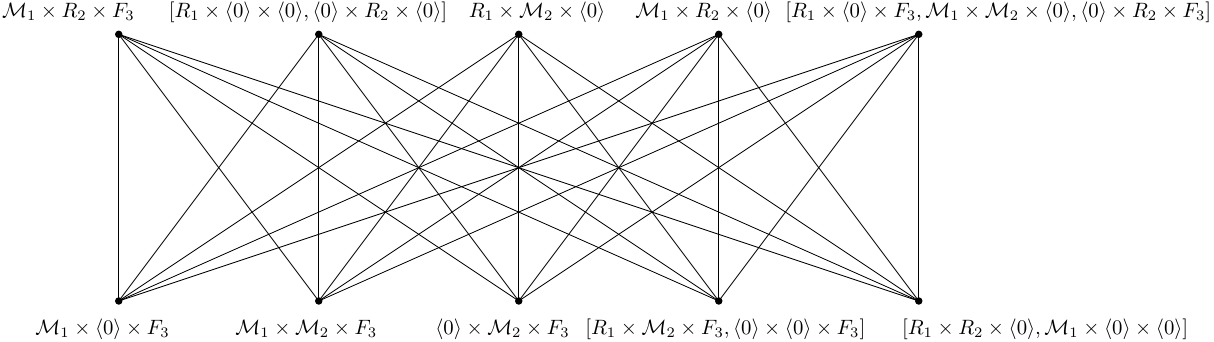}
\caption{Subgraph of $\text{PIS}(R_1 \times R_2 \times F_3)$}
\label{K_55subdivisionR_1R_2F_3fig1}
\end{figure}
    Now let, $R \cong R_1 \times F_2 \times F_3$. Further, first suppose that $R_1$ is not a principal ring and consider $\mathcal{M}_1 = \langle x,y \rangle$. Consider the ideals $I_1 = \langle x \rangle$, $ I_2 = \langle y \rangle$ and $I_3 = \langle x + y \rangle$ of $R_1$. Then by Figure \ref{K_55SubdivisionR_1F_2F_3fig2}, $\textnormal{PIS}(R)$ contains a subgraph homeomorphic to $K_{5,5}$. Therefore, $g(\textnormal{PIS}(R_1 \times F_2 \times F_3)) \ge 3$.
    \begin{figure}[h!]
\centering
\includegraphics[width=0.7 \textwidth]{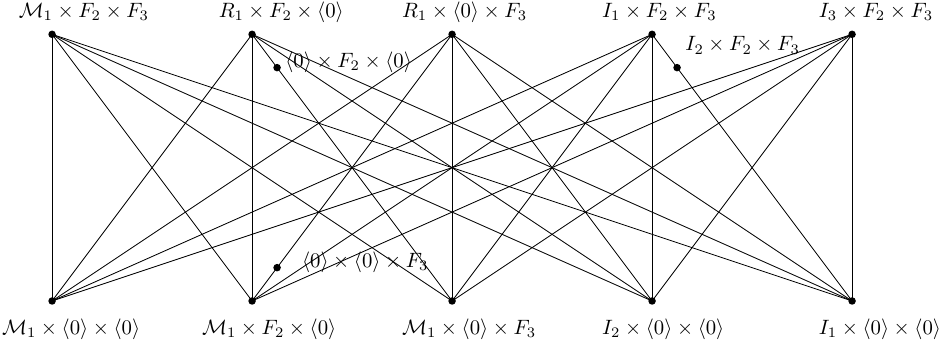}
\caption{Subgraph of $\text{PIS}(R_1 \times F_2 \times F_3)$, where $\mathcal{M}_1 = \langle x, y \rangle$}
\label{K_55SubdivisionR_1F_2F_3fig2}
\end{figure}
    Now suppose that $R_1$ is a principal ring. It implies that $\mathcal{M}_1$ is a principal ideal. 
    If $\eta(\mathcal{M}_1) \le 2$, then by Theorem \ref{Planar_primeidealsum} and Theorem \ref{genus13product}, we have $g(\textnormal{PIS}(R)) \le 1$. Now let $R \cong R_1 \times F_2 \times F_3$ such that $\eta(\mathcal{M}_1) = 3$. Then  $\mathcal{I}^{*}(R_1) = \{\mathcal{M}_1, \mathcal{M}_1^2 \}$. Note that $\textnormal{PIS}(R)$ contains a subgraph $G_2$ homeomorphic to two blocks of $K_{3,3}$ with commom vertex $\mathcal{M}_1 \times F_2 \times F_3$ (see Figure \ref{blockK_33primeidealsumfig4}).
            \begin{figure}[h!]
\centering
\includegraphics[width=0.65 \textwidth]{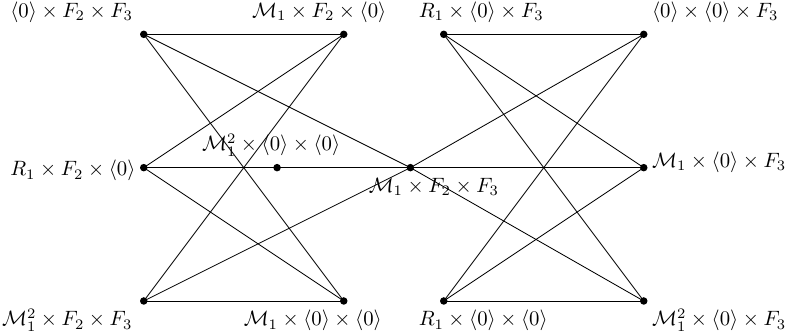}
\caption{Subgraph $G_2$ of $\text{PIS}(R_1 \times F_2 \times F_3)$}
\label{blockK_33primeidealsumfig4}
\end{figure}
    By Lemma \ref{genusofblocks}, it implies that $g(G_2) =2$ and so $g(\textnormal{PIS}(R_1 \times F_2 \times F_3)) \ge 2$. Suppose that $g(\textnormal{PIS}(R_1 \times F_2 \times F_3)) = 2$. Now to embed $\textnormal{PIS}(R_1 \times F_2 \times F_3)$ from $G_2$, first we insert the vertices $\mathcal{M}_1^2 \times F_2 \times \langle 0 \rangle$ and $\langle 0 \rangle \times F_2 \times \langle 0 \rangle$ in an embedding of $G_2$ in $\mathbb{S}_2$. Since both the vertices $u= \mathcal{M}_1^2 \times F_2 \times \langle 0 \rangle$ and $v = \langle 0 \rangle \times F_2 \times \langle 0 \rangle$ are adjacent to the vertices $R_1 \times \langle 0 \rangle \times \langle 0 \rangle $, $R_1 \times F_2 \times \langle 0 \rangle $ and $\mathcal{M}_1 \times F_2 \times F_3$, therefore $u$ and $v$ must be inserted in two different faces $F$ and $F'$ having $R_1 \times \langle 0 \rangle \times \langle 0 \rangle $, $R_1 \times F_2 \times \langle 0 \rangle $ and $\mathcal{M}_1 \times F_2 \times F_3$ as common vertices. This is not possible without edge crossings because none of the vertices $R_1 \times \langle 0 \rangle \times \langle 0 \rangle $, $R_1 \times F_2 \times \langle 0 \rangle $ and $\mathcal{M}_1 \times F_2 \times F_3$ are of degree two in $G_2$.
    Consequently, $g(\textnormal{PIS}(R_1 \times F_2 \times F_3)) \ge 3$. This completes our proof.
 \end{proof}

\noindent\textbf{Proof of Theorem \ref{genus2R_1R_2}:}
   First suppose that $g(\textnormal{PIS}(R)) = 2$. By Proposition \ref{genus2ofnproduct}, we have $R \cong R_1 \times R_2$. Assume that $R_1$ is not a principal ring and consider $\mathcal{M}_1 = \langle x,y \rangle$. Further, suppose that $R_2$ has exactly one nonzero proper ideal $\mathcal{M}_2$. Consider the ideals $I_1 = \langle x \rangle$, $ I_2 = \langle y \rangle$ and $I_3 = \langle x + y \rangle$ of $R_1$. Then $\textnormal{PIS}(R)$ contains a subgraph $G'$ homeomorphic to $K_{5,4}$ (see Figure \ref{K_54SubdivisionR_1R_2fig5}).
            \begin{figure}[h!]
\centering
\includegraphics[width=0.65 \textwidth]{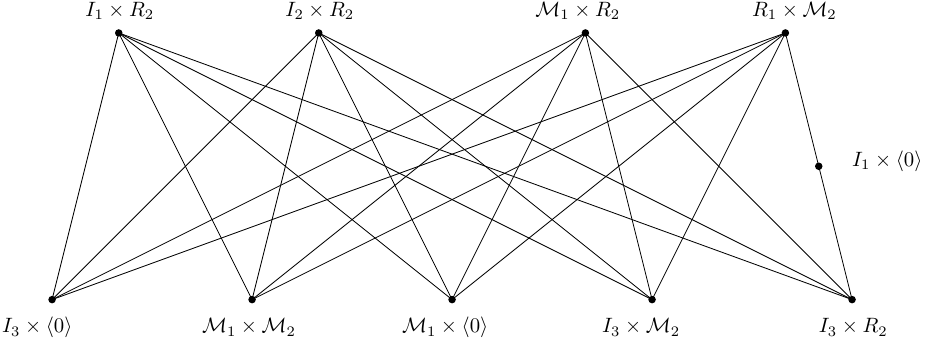}
\caption{Subgraph $G'$ of $\text{PIS}(R_1 \times R_2 )$}
\label{K_54SubdivisionR_1R_2fig5}
\end{figure}
    It implies that $g(G') = 2$ and so $g(\textnormal{PIS}(R)) \ge 2$. Suppose that $g(\textnormal{PIS}(R)) = 2$. Now to embed $\textnormal{PIS}(R)$ from $G'$, first we insert the vertices $I_1 \times \mathcal{M}_2$, $R_1 \times \langle 0 \rangle$ and $I_2 \times \mathcal{M}_2$ in an embedding of $G'$ in $\mathbb{S}_2$. Since $I_1 \times \mathcal{M}_2 \sim R_1 \times \langle 0 \rangle \sim I_2 \times \mathcal{M}_2$, therefore all these vertices must be inserted in the same face $F$. Note that, both the vertices $u' = I_1 \times \mathcal{M}_2$ and $ v' = I_2 \times \mathcal{M}_2$ are adjacent with $I_3 \times R_2$, $\mathcal{M}_1 \times R_2$ and $R_1 \times \mathcal{M}_2$. It implies that the vertices $u'$ and $v'$ cannot be inserted in $F$ without edge crossings. Consequently, $R_2$ is a field and so $R \cong R_1 \times F_2$. Now assume that $\mathcal{M}_1 = \langle x,y,z \rangle$. Consider the ideals $I_1 = \langle x \rangle$, $ I_2 = \langle y \rangle$, $I_3 = \langle z \rangle$, $I_4 = \langle x, y \rangle$, $I_5 = \langle y,z \rangle$, $I_6 = \langle x,z \rangle$ and $I_7 = \langle y,  x + z \rangle$ of $R_1$. Then $\textnormal{PIS}(R)$ contains a subgraph homeomorphic to $K_{5,5}$ (see Figure \ref{K_55SubdivisionR_1R_2fig6}), a contradiction.
\begin{figure}[h!]
\centering
\includegraphics[width=0.65 \textwidth]{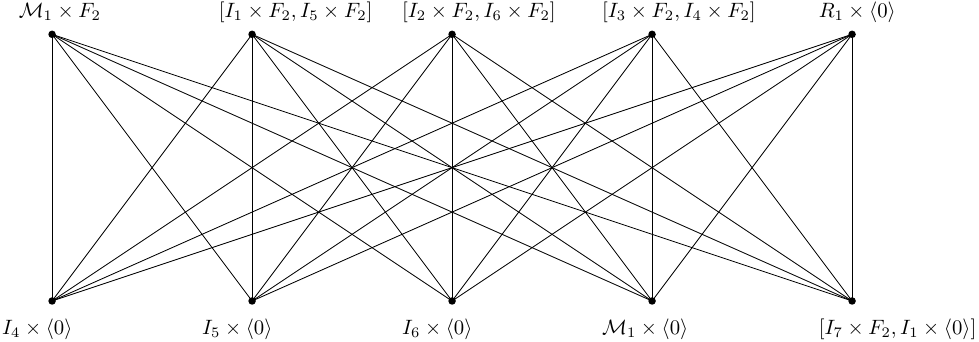}
\caption{Subgraph of $\text{PIS}(R_1 \times F_2 )$, where $\mathcal{M}_1 = \langle x, y, z \rangle$}
\label{K_55SubdivisionR_1R_2fig6}
\end{figure}
    It follows that $\mathcal{M}_1 = \langle x,y \rangle$. Suppose that $x^2 \neq 0$. Consider the ideals $I_1 = \langle x \rangle$, $ I_2 = \langle y \rangle$,  $I_3 = \langle x + y \rangle$, $I_4 = \langle x^2, x + y \rangle$ and $I_5 = \langle x^2, y \rangle$ of $R_1$. Then $\textnormal{PIS}(R)$ contains a subgraph homeomorphic to $K_{5,5}$ (see Figure \ref{K_55SubdivisionR_1F_2fig7}), a contradiction. It implies that $x^2 = y^2 =0$.
\begin{figure}[h!]
\centering
\includegraphics[width=0.65 \textwidth]{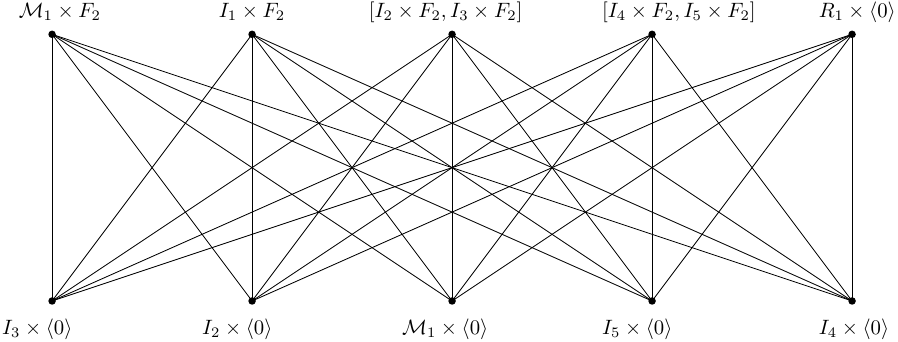}
\caption{Subgraph of $\text{PIS}(R_1 \times F_2 )$, where $\mathcal{M}_1 = \langle x, y \rangle$ }
\label{K_55SubdivisionR_1F_2fig7}
\end{figure}

    Next assume that both $R_1$ and $R_2$ are principal rings. If one of $R_i$ is a field, then by Theorem \ref{Planar_primeidealsum}, we obtain $g(\textnormal{PIS}(R)) = 0$, a contradiction. Let $\eta(\mathcal{M}_2)= 2$. If $\eta(\mathcal{M}_1) \le 3$, then by Theorem \ref{Planar_primeidealsum} and Theorem \ref{genusofR1R2}, we obtain $g(\textnormal{PIS}(R)) \le 1$, again a contradiction. Assume that $\eta(\mathcal{M}_1) \ge 6$. Then by Figure \ref{K_33subdivisionfig9}, note that $\textnormal{PIS}(R)$ contains a subgraph $H$ homeomorphic to two blocks of $K_{3,3}$ with common vertex $\mathcal{M}_1 \times R_2$.
    \begin{figure}[h!]
\centering
\includegraphics[width=0.65 \textwidth]{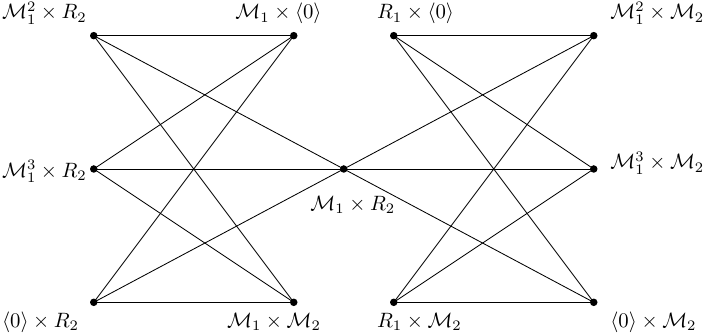}
\caption{Subgraph $H$ of $\text{PIS}(R_1 \times R_2 )$ }
\label{K_33subdivisionfig9}
\end{figure}
It follows that $g(H) = 2$ and so $g(\textnormal{PIS}(R)) \ge 2$. Suppose that $g(\textnormal{PIS}(R)) = 2$. Now to embed $\textnormal{PIS}(R)$ in $\mathbb{S}_2$ from $H$, first we insert the vertices $\mathcal{M}_1^4 \times R_2$, $\mathcal{M}_1^5 \times R_2$ and their incident edges in an embedding of $H$ in $\mathbb{S}_2$. Note that both the vertices $ u = \mathcal{M}_1^4 \times R_2$ and $ v = \mathcal{M}_1^5 \times R_2$ are adjacent with the vertices $\mathcal{M}_1 \times \langle 0 \rangle$, $\mathcal{M}_1 \times \mathcal{M}_2$ and $\mathcal{M}_1 \times R_2$. It implies that $u$ and $v$ must be inserted into two different faces having common vertices as $\mathcal{M}_1 \times \langle 0 \rangle$, $\mathcal{M}_1 \times \mathcal{M}_2$ and $\mathcal{M}_1 \times R_2$, which is not possible without edge crossings because none of the vertices $\mathcal{M}_1 \times \langle 0 \rangle$, $\mathcal{M}_1 \times \mathcal{M}_2$ and $\mathcal{M}_1 \times R_2$ are of degree two in $H$. Therefore,  $g(\textnormal{PIS}(R)) \ge 3$, a contradiction. Consequently, $\eta(\mathcal{M}_1) \in \{ 4,5\}$.

Next let $\eta(\mathcal{M}_2)= 3$. To prove our result it is sufficient to show that $\eta(\mathcal{M}_1) \neq 3$. On contrary assume that $\eta(\mathcal{M}_1)= 3$. Then by Figure \ref{K_33subdivisionfig8}, note that $\textnormal{PIS}(R)$ contains a subgraph $G$, which contains a subgraph homeomorphic to $K_{3,3}$.
   \begin{figure}[h!]
\centering
\includegraphics[width=0.65 \textwidth]{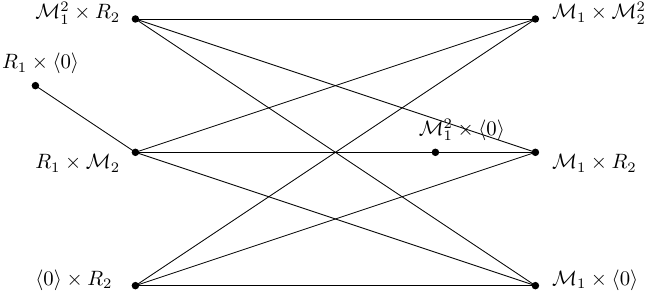}
\caption{Subgraph $G$ of $\text{PIS}(R_1 \times R_2 )$ }
\label{K_33subdivisionfig8}
\end{figure}
It follows that $g(G) = 1$ and so $g(\textnormal{PIS}(R)) \ge 1$. Suppose that $g(\textnormal{PIS}(R)) = 1$. Now to embed $\textnormal{PIS}(R)$ in $\mathbb{S}_1$ from $G$, first we insert the vertices $\langle 0 \rangle  \times \mathcal{M}_2$, $\mathcal{M}_1^2 \times \mathcal{M}_2$, $\mathcal{M}_1  \times \mathcal{M}_2$ and their incident edges in an embedding of $G$ in $\mathbb{S}_1$. Note that the vertices $\langle 0 \rangle  \times \mathcal{M}_2$, $\mathcal{M}_1^2 \times \mathcal{M}_2$ and $\mathcal{M}_1  \times \mathcal{M}_2$ are adjacent with the vertices $R_1 \times \langle 0 \rangle$, $R_1 \times \mathcal{M}_2$ and $\mathcal{M}_1 \times R_2$. Also $R_1 \times \langle 0 \rangle$ is of degree one in $G$, so it must lie inside some face $F$ of an embedding of $G$ in $\mathbb{S}_1$. It implies that the vertices $\langle 0 \rangle  \times \mathcal{M}_2$, $\mathcal{M}_1^2 \times \mathcal{M}_2$ and $\mathcal{M}_1  \times \mathcal{M}_2$ must be inserted in $F$, which is not possible without edge crossings. Now let $G'$ be the graph induced by the vertex set $V(G) \cup \{ \langle 0 \rangle  \times \mathcal{M}_2, \mathcal{M}_1^2 \times \mathcal{M}_2, \mathcal{M}_1  \times \mathcal{M}_2 \} $. It implies that $g(G') \ge 2$ and so $g(\textnormal{PIS}(R)) \ge 2$. Suppose $g(\textnormal{PIS}(R)) = 2$. To embed $\textnormal{PIS}(R)$ in $\mathbb{S}_2$ from $G'$, first we insert the vertices $R_1 \times \langle 0 \rangle $, $R_1 \times \mathcal{M}_2^2$ and their incident edges in an embedding of $G'$ in $\mathbb{S}_2$. Note that both the vertices $u_1 = R_1 \times \langle 0 \rangle $ and $v_1 = R_1 \times \mathcal{M}_2^2$ are adjacent with the vertices $\langle 0 \rangle  \times \mathcal{M}_2$, $\mathcal{M}_1^2 \times \mathcal{M}_2$ and $\mathcal{M}_1  \times \mathcal{M}_2$. It implies that $u_1 $ and $v_1$ must be inserted in the different faces containing the vertices $\langle 0 \rangle  \times \mathcal{M}_2$, $\mathcal{M}_1^2 \times \mathcal{M}_2$ and $\mathcal{M}_1  \times \mathcal{M}_2$, which is not possible without edge crossings because none of the vertices $\langle 0 \rangle  \times \mathcal{M}_2$, $\mathcal{M}_1^2 \times \mathcal{M}_2$ and $\mathcal{M}_1  \times \mathcal{M}_2$ are of degree two in $G'$. Therefore, $g(\textnormal{PIS}(R)) \geq 3$, a contradiction.

    Converse follows from the Figures \ref{genus2fig1}, \ref{genus2fig2} and \ref{genus2fig3}.
            \begin{figure}[h!]
\centering
\includegraphics[width=0.63 \textwidth]{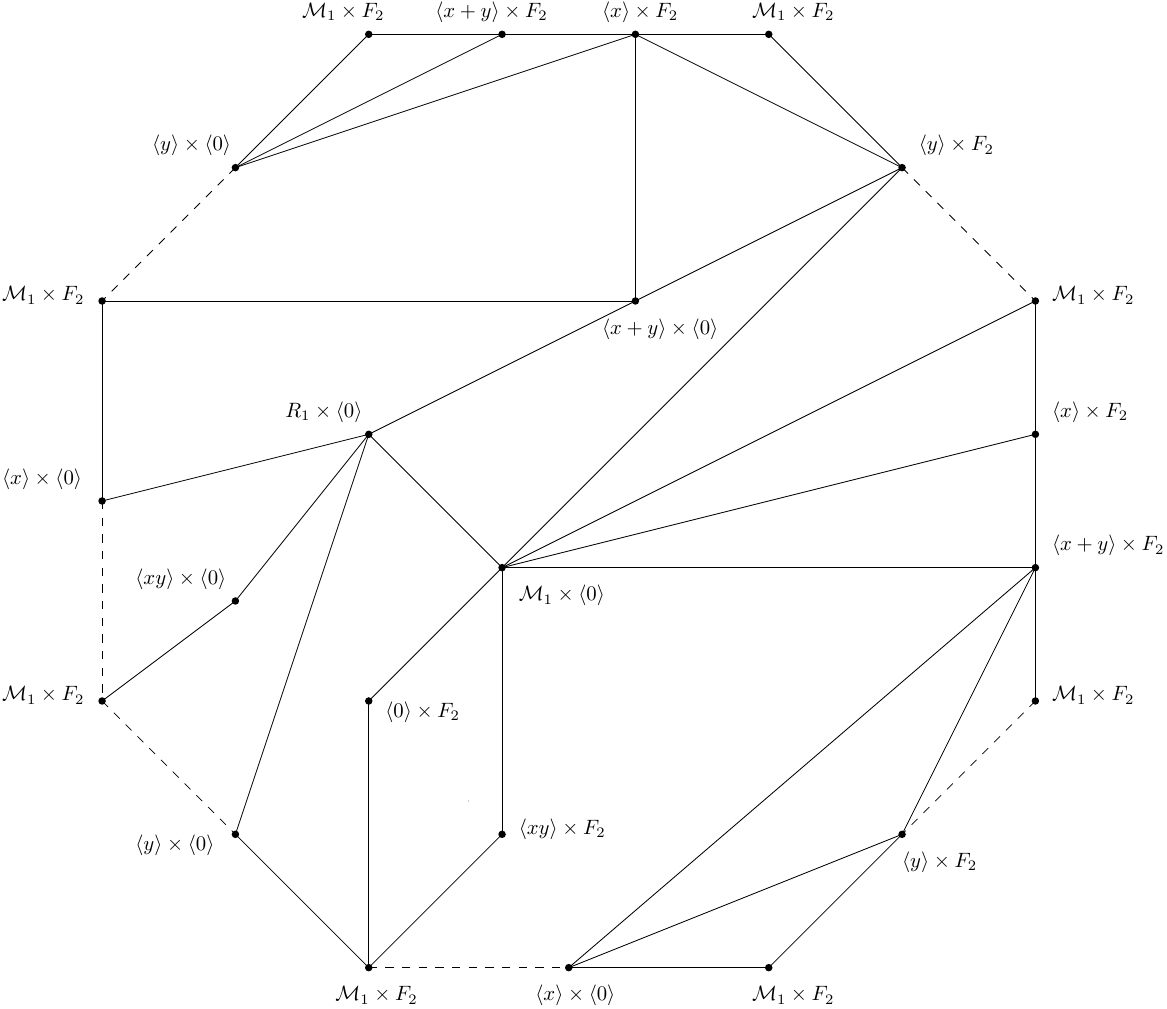}
\caption{Embedding of $\textnormal{PIS}(R_1 \times F_2)$ in $\mathbb{S}_2$, where $\mathcal{M}_1=  \langle x,y \rangle$ such that $x^2 = y^2 =0$.}
\label{genus2fig1}
\end{figure}

\begin{figure}[h!]
\centering
\includegraphics[width=0.63 \textwidth]{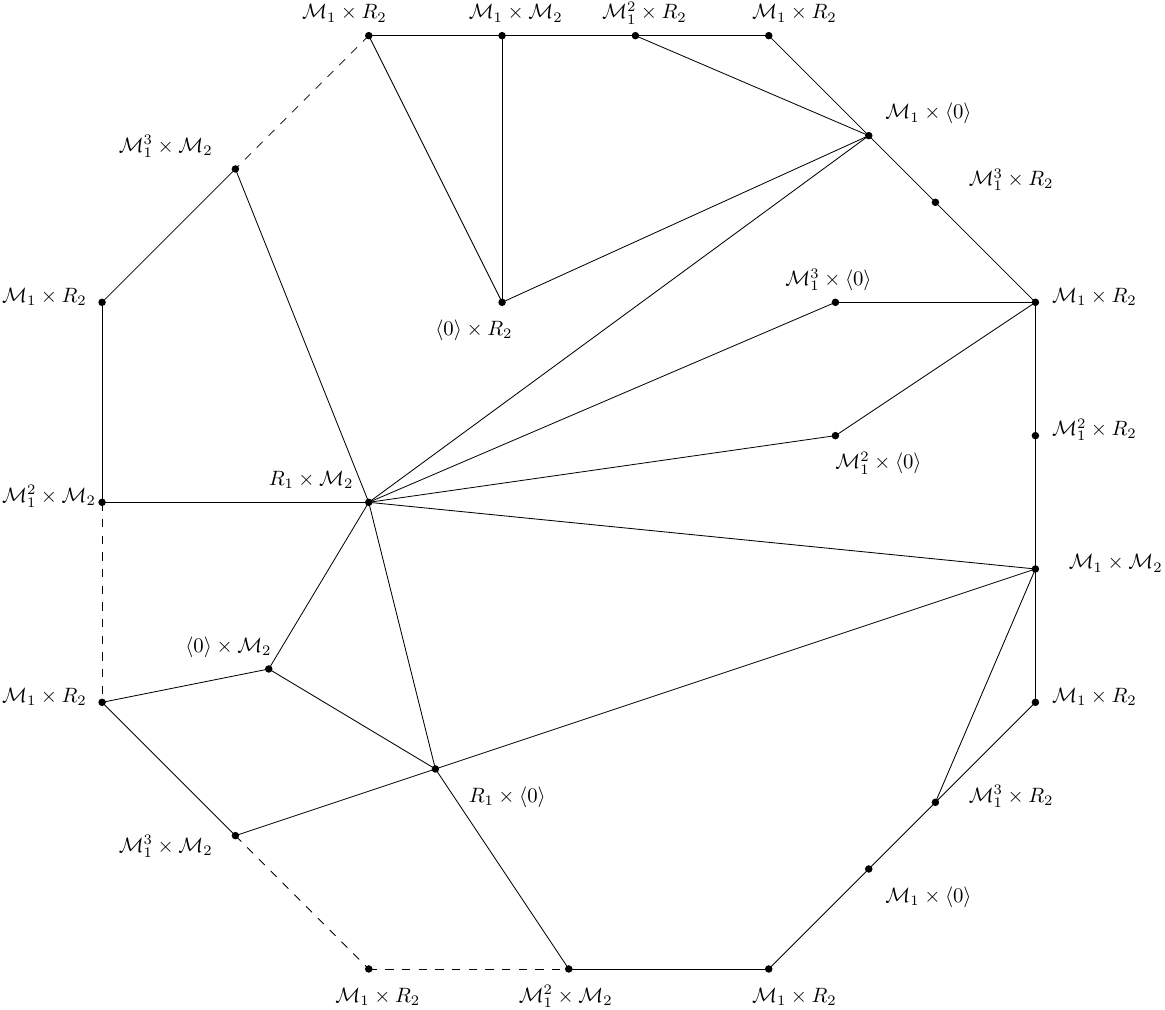}
\caption{Embedding of $\textnormal{PIS}(R_1 \times R_2)$ in $\mathbb{S}_2$, where $\eta(\mathcal{M}_1)= 4$ and $\eta(\mathcal{M}_2)= 2$.}
\label{genus2fig2}
\end{figure}

\begin{figure}[h!]
\centering
\includegraphics[width=0.63 \textwidth]{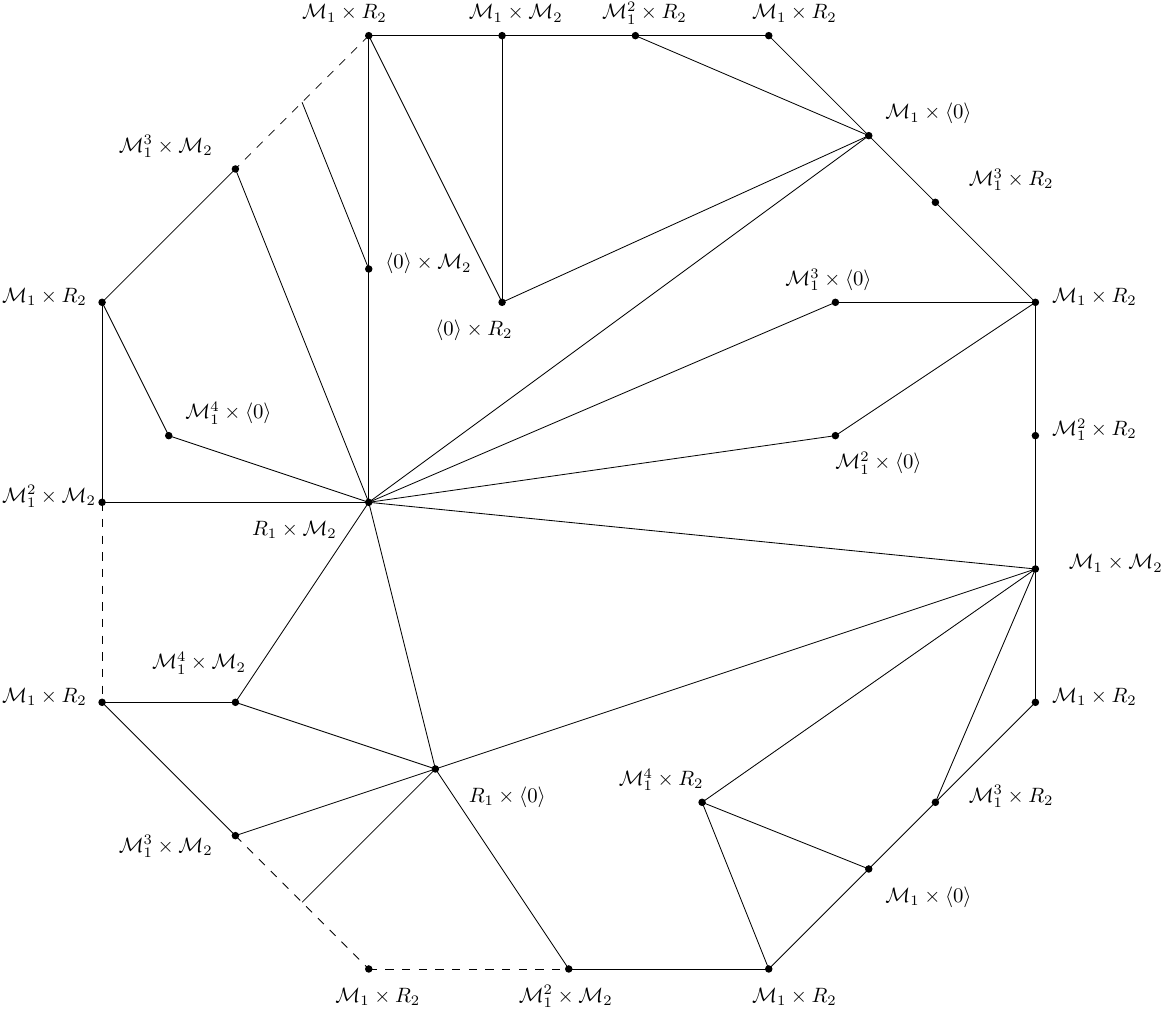}
\caption{Embedding of $\textnormal{PIS}(R_1 \times R_2)$ in $\mathbb{S}_2$, where $\eta(\mathcal{M}_1)= 5$ and $\eta(\mathcal{M}_2)= 2$}
\label{genus2fig3}
\end{figure}

%%%%%%%%%%%%%%%%%%%%%%%%%%%%%%%%%%%%%%%%%%%%%%%%%%%%%%%%%%%%%%%%%%%%%%%%%%%%%%%%%%%%%%%%%%%%%%%%%%%%%%%%%%%%%%%%%%%%%%%%%%%%%%%%%%%%%%%%%%%%%%%%%%%%%%%%%%

  \newpage 
\textbf{Acknowledgement:} We would like to thank the referee for his/her valuable suggestions which helped us to improve the presentation of the paper. The first author gratefully acknowledge for providing financial support to Birla Institute of Technology and Science (BITS) Pilani, Pilani-333031, India.   

% \vspace{.3cm}
% \textbf{Disclosure statement:} There is no conflict of interest regarding the publishing of this paper. 

 %\bibliographystyle{abbrv}
%	\bibliography{ref}

 \end{document}